\documentclass[12pt]{article}
\usepackage{amssymb, amsmath, amsthm, epsf, epsfig, color}

\newcommand{\E}{{\mathbf E}}
\renewcommand{\P}{{\mathbf P}}
\newcommand{\sig}{\sigma}

\newcommand{\be}{\begin{equation}}
\newcommand{\ee}{\end{equation}}
\newcommand{\bea}{\begin{eqnarray}}
\newcommand{\eea}{\end{eqnarray}}

\newlength{\originalbase}
\newcommand{\spacing}[1]{\setlength{\baselineskip}{#1\originalbase}}
\newcommand{\ep}{\varepsilon}

\renewcommand{\t}{T_G}
\newcommand{\B}{{\mathcal B}}

\newcommand{\Z}{\mathbb{Z}}
\newcommand{\R}{\mathbb{R}}

\def\eps{\epsilon}
\begin{document}
\spacing{1.5}
\newtheorem{theorem}{Theorem}[section]
\newtheorem{claim}{Claim}[theorem]
\newtheorem{prop}[theorem]{Proposition}
\newtheorem{remark}[theorem]{Remark}
\newtheorem{lemma}[theorem]{Lemma}
\newtheorem{corollary}[theorem]{Corollary}
\newtheorem{guess}[theorem]{Conjecture}
\newtheorem{conjecture}[theorem]{Conjecture}

%\title{Single Site versus Block Glauber Dynamics}
\title{Can extra updates delay mixing?}

\author{
Yuval Peres~\thanks{Microsoft Research, Redmond WA 98052}
\and Peter Winkler~\thanks{Dept.\ of Mathematics, Dartmouth College, Hanover, NH 03755.}
}

\maketitle

\begin{abstract}
We consider Glauber dynamics (starting from an extremal configuration)
in a monotone spin system, and show that interjecting extra updates
cannot increase the expected Hamming distance or the total variation distance to the stationary
distribution.  We deduce that for monotone Markov random fields,
when block dynamics contracts a Hamming metric,
  single-site dynamics mixes in O$(n \log n)$ steps on an $n$-vertex graph.
In particular, our result completes work of Kenyon, Mossel and Peres concerning
Glauber dynamics for the Ising model on trees.  Our approach also shows that on bipartite graphs,
alternating updates systematically between odd and even vertices cannot improve the mixing time
by more than a factor of $\log n$ compared to updates at uniform random locations on an $n$-vertex graph.
Our result is especially effective in comparing block and single-site dynamics; it has already been
used in works of Martinelli, Sinclair, Mossel, Sly, Ding, Lubetzky, and Peres in various combinations.
\end{abstract}

\medskip

\section{Introduction}

In a number of cases, mixing rates have been determined for Glauber
dynamics using block updates, but only rough estimates have been
obtained for single site dynamics.  Examples include the Ising model
on trees and the monomer-dimer model on $\Z^d$.  In this work, we
employ a ``censoring lemma'' for monotone systems to transport
bounds for block dynamics to bounds for single site dynamics;
sharp estimates result in several situations.

Our main interest is in spin systems
with nearest-neighbor interactions on a finite graph $G$.
A configuration consists of a mapping $\sigma$ from the set $V$ of sites
of $G$ to a fixed partially ordered set $S$ of ``spins''.  The probability
$\pi(\sigma)$ of a configuration $\sigma$ is given by
$$
\frac{1}{Z} \prod_{u \sim v} \Psi(\sigma_u,\sigma_v)
$$
where $Z$ is the appropriate normalizing constant.
More generally, our results apply when $\pi$ defines a
 monotone Markov random field.
In single site Glauber dynamics, at each step, a uniformly random
site is ``updated'' and assumes a new spin according to $\pi$ conditioned on
the spins of its neighbors.  The resulting Markov chain is irreducible, aperiodic, and
has unique stationary distribution $\pi$.
Let $p^t(\omega, \cdot)$ be
the distribution of configurations after $t$ steps, with initial state
$\omega$.  Let
$\|\mu-\nu\|=\frac12\sum_{\sigma}|\mu(\sigma)-\nu(\sigma)|$ be the
total variation distance.
The {\bf mixing time} $\t(\eps)$ for the dynamics is the
least $t$ such that $\|p^t(\omega, \cdot)-\pi \| \le \eps$ for any
 $\omega \in \Omega$.
In discrete-time block dynamics, a family $\B$ of ``blocks'' of sites is
provided.
At each step, a block $B \in \B$ is selected uniformly at random
and a configuration on $B$ is selected according to $\pi$ conditioned on
the spins of the sites in the exterior boundary of $B$.
A useful method of bounding mixing times is to first bound the spectral gap of the block dynamics
using path coupling, and then use comparison theorems for the spectral gap to derive a bound for $\t(\eps)$.
In key examples of Glauber dynamics for the Ising model on lattices and trees,
this method tends to overestimate  $\t(\eps)$ by a factor of $n$ on an $n$-vertex graph.

Stated informally,
our main results are:
\begin{itemize}
\item In Glauber dynamics for a monotone (i.e., attractive) spin system,
started at the top or bottom state, censoring updates increases the distance from stationarity.
\item Suppose a monotone spin system on an $n$-vertex graph $G$
has a block dynamics which contracts
(on average) a Hamming metric, and single-site dynamics on each block
with arbitrary boundary conditions mixes in a bounded time.
If the collection of blocks can be partitioned into a bounded number of layers such that
blocks in each layer are nonadjacent, and weights within a block have a bounded ratio,
then discrete time single site dynamics on $G$ mixes (in total variation) in $O(n \log n)$ steps.
\item In \cite{KMP} (see also \cite{BKMP}) it was proved that
for the Ising model on an $n$-vertex $b$-ary tree, block dynamics
with large bounded blocks contracts a (weighted) Hamming metric at temperatures above
the extremality threshold. This, in conjunction with our main results, implies
that single-site dynamics on these trees mixes in $O(n \log n)$ steps.
(See \cite{MSW} for refinements of this theorem using Log-Sobolev inequalities).
\item
If $H$ is a subgraph of $G$ and only one vertex in $H$ is adjacent to
vertices in $G \setminus H$, then continuous-time Glauber dynamics on $H$
mixes faster than the restriction to $H$ of continuous-time Glauber
dynamics on $G$.
\item For an $n$ vertex bipartite graph, alternating updating
of all the ``odd'' and all the ``even'' vertices cannot mix much faster
than systematic updates (enumerating the vertices in an arbitrary order):
The odd-even updates can reduce the number of
vertices updated at most by a factor of two. Similarly,
the odd-even updates can be faster than uniformly random updates by a factor of at most $\log n$.
%\item
\end{itemize}
See \S 1.2 for further discussion of block dynamics, and \S 2-3 for proofs. A preliminary version of our results, including the proof of Theorem \ref{subseq}, was presented in the 2005 lectures \cite{P}.

\subsection{Terminology}
In what follows, a {\em system} $\langle \Omega,S,V,\pi \rangle$ consists
of a finite set $S$ of spins, a set $V$ of sites,
a space $\Omega \subseteq S^V$ of configurations (assignments of spins
to sites), and a distribution $\pi$ on $\Omega$, which will serve as the
stationary distribution for our Glauber dynamics.
We assume that $\pi(\omega)>0$ for $\omega \in \Omega$.
The Ising model
(where $S=\{+,-\}$ and $\Omega = S^V$) is the basic example; we allow
$\Omega$ to be a strict subset of $S^V$ to account for ``hard constraints''
such as those imposed by the hard-core gas model.

We denote by $\sigma_v^s$ the configuration
obtained from $\sigma$ by changing its value at $v$ to $s$, that is,
$\sigma_v^s(v)=s$ and $\sigma_v^s(u)=\sigma(u)$ for all $u \not= v$.
Let $\sigma_v^\bullet$ be the set of configurations $\{\sigma_v^s\}_{s \in S}$
in $\Omega$.

The {\bf update} $\mu_v$ at $v$ of a distribution
$\mu$ on $\Omega$ is defined by
\begin{equation} \label{update}
\mu_v(\sigma) =
\frac{\pi(\sigma)}{\pi(\sigma_v^\bullet)} \mu(\sigma_v^\bullet)
\quad \mbox{ \rm for } \; \sigma \in \Omega.
\end{equation}

For measures $\mu$ and $\nu$ on a poset $\Gamma$, we write
$\nu \preceq \mu$ to
indicate that $\mu$ {\em stochastically dominates\/} $\nu$, that is,
$\int g\,d\nu \le \int g\,d\mu$ for all increasing
functions $g: \Gamma \to \R$.

The system $\langle \Omega,S,V,\pi \rangle$ is called {\bf monotone}
if $S$ is totally ordered, $S^V$ is endowed
with the coordinate-wise partial order, and
whenever $\sigma, \, \tau \in \Omega$ satisfy $\sigma \le \tau$,
then for any vertex $v \in V$ we have
\begin{equation} \label{mono}
\Big\{\frac{\pi(\sigma_v^s)}{\pi(\sigma_v^\bullet)}\Big\}_{s \in S}
\preceq \, \, \Big\{\frac{\pi(\tau_v^s)}{\pi(\tau_v^\bullet)}\Big\}_{s \in S}
\end{equation}
as distributions on the spin set $S$.

\subsection{Main results}
\begin{theorem} \label{subseq}
Let $\langle \Omega,S,V,\pi \rangle$ be a monotone system and
%with totally ordered finite spin set $S$.
let $\mu$
be the distribution on $\Omega$ which results from successive updates
at sites $v_1,\dots,v_m$, beginning at the top configuration.
Let $\nu$ be defined similarly but with updates
only at a subsequence $v_{i_1},\ldots,v_{i_k}$.
Then
%\begin{enumerate}
 $\mu \preceq \nu$, and
$\|\mu - \pi\| \leq \|\nu - \pi\|$ in total variation. Moreover, this also holds if the sequence
$v_1,\dots,v_m$ and its subsequence $i_1,\ldots,i_k$ are chosen at random according to any prescribed distribution.
%\end{enumerate}
\end{theorem}
See \S\ref{pf-censor} for the proof, which shows also that the assumption of starting from the top configuration
can be replaced by the assumption that the dynamics starts at a distribution $\mu_0$ where the likelihood ratio $\mu_0/\pi$ is weakly increasing.
Other assumptions, in particular monotonicity of the system, cannot be dispensed with, as shown recently by Holroyd \cite{Ho}.

\medskip

Next, we discuss block dynamics and the contraction method to bound mixing times for spin systems.

Let us endow $\Omega \subset S^V$ with the Hamming metric $H(\sigma, \tau)= |\{v \in V: \sigma_v \not= \tau_v\}$.
(More generally, it is sometimes fruitful to consider a weighted $\ell^1$ metric).
The {\em Kantorovich distance} $\rho(\mu,\nu)$ between two distributions on $\Omega$
is defined to be the minimum over all couplings of $\mu$ and $\nu$ of $\E H(\sigma,\tau)$,
where $\sigma$ is drawn from $\mu$ and $\tau$ from $\nu$. The fact that this metric satisfies the triangle inequality is proved, e.g., in Chapter 14 of \cite{LPW} and is essentially equivalent to the path-coupling Theorem of \cite{BD}.

Given a subset $B$ of $V$, let $\sigma_B^\bullet$ be the set of configurations
$\tau \in \Omega$ such that $\tau$ agrees with $\sigma$ on $V \setminus B$.
For $\sigma \in \Omega$, the {\em block update} $U_B\sigma$ is a measure on $\sigma_B^\bullet$
defined by $(U_B\sigma)(\omega) = \frac{\pi(\omega)}{\pi(\sigma_B^\bullet)}$ for $\omega \in \sigma_B^\bullet$.
Thus $U_B\sigma$ is $\pi$ conditioned on  $\sigma_B^\bullet$.
For a collection of blocks $\B$, the $\B$-{\em averaged block
update} of $\sigma \in \Omega$ yields a random configuration with distribution $\frac{1}{|\B|} \sum_{B \in \B} U_B\sig \,.$
The {\em block dynamics\/} determined by $\B$ consists of performing successive $\B$-averaged block
updates.

%In general, a dynamics on a graph with $N$ nodes is said to be {\em eventually contracting} if application of the dynamics for
%$CN$ steps (where $C$ is a suitable constant)  reduces the expected Hamming
%distance between any two configurations by a constant factor.
We say that a block dynamics is {\em contracting\/} if for any two configurations $\sigma$ and $\tau$, the expected number
of discrepancies after a block update is smaller by a factor of $1{-}\gamma|B|/|V|$ or less,
where $\gamma$ is a constant and $|B|$ is the number of sites in a block.  The triangle inequality for the Kantorovich metric implies that it suffices to verify this contraction condition when $\sigma$ and $\tau$ differ at a single site.  In our setting, contraction implies a bound of order $|V|\log|V|$ on the mixing time, since the number of blocks is of order $|V|$.   When the blocks are cubes in a lattice, a sufficient condition for contraction of block dynamics is {\em strong spatial mixing}, as defined and studied in \cite{MO, Ma, Ma3, DSVW}.

The system $\langle \Omega,S,V,\pi \rangle$ is a {\bf Markov random field} if for any set $B \subset V$
and $\sigma \in \Omega$, the distribution $U_B\sigma$
%\begin{equation} \label{mrf}
%\pi\Bigl(\sigma \, \Big|\; \sigma|_{B^c} =\eta \, \Bigr)
%%\pi\Bigl(\sigma \, \Big|\; \sigma \mbox{ \rm  on }  \partial B \, \Bigr)\,.
%\end{equation}
depends only on the restriction of $\sigma$ to
$\partial B$, the set of vertices in $B^c$ that are adjacent to $B$.

The next theorem is intended to illustrate how, in a particular case,
Theorem~\ref{subseq} can be used to deduce rapid mixing for single-site
dynamics from a contraction condition for block dynamics.

\begin{theorem}
Let $\Omega$ be the configuration space for a monotone Markov random field on the $d$-dimensional toroidal grid $V = [0,N{-}1]^d$.
Let $(\ell{+}1)|N$ and for each $v \in V$, let $B_v$ be the cube of side-length $\ell$ anchored
at $v$.  If the corresponding block dynamics is contracting, and the single-site dynamics
restricted to any block has uniformly bounded mixing time (for all boundary conditions), then single-site dynamics on all
of $V$ has mixing time ${\rm O}(|V|\log|V|)$, where the implied constant depends only on the contraction parameter $\gamma$ and on $\ell$.
\end{theorem}

\begin{proof}
For any $u \in V$ and any block $B$, let
$$
\Phi_u(B) = \max_{\sigma \in \Omega, s \in S} \rho(U_B\sigma, U_B\sigma')
$$
where $U_B\sigma$ is the distribution that results when $B$ is updated from configuration $\sigma$,
and $\sigma'=\sigma_u^s$ is obtained from $\sigma$ by changing the spin at $u$ to $s$.  Since
$H(\sigma,\sigma')=1$, we have $\Phi_u(B)=1$ when neither $B$ nor $\partial B$ contains $u$.
If $u \in B$, then $\Phi_u(B)=0$, so the key case is when $u$ is on the exterior boundary of $B$.

Since the dynamics for updating a random block $B$ is assumed to be contracting, we have
in particular that for some constant $\gamma$, and any $u \in V$,
\begin{equation}\label{eqn:ham}
\gamma \ell^d/N < \E\Delta = \P(B \owns u) - \frac1N \sum_{\partial B \owns u} \Phi_u(B)
  = \frac{\ell^d}{N} - \frac1N \sum_{\partial B \owns u} \Phi_u(B)
\end{equation}
where $\Delta$ is the decrease in Hamming distance between $\sigma$ and $\sigma'$ caused by the update.

Let $t$ be the number of single-site updates, performed uniformly at random on the sites
{\em inside} a box $B$, needed (regardless of boundary spins) to bring the Kantorovich
distance between the resulting configuration on $B$ and the block-update configuration
down to at most $\delta$, where $\delta = \gamma \ell^d/(4|\partial B|) = \gamma/(1+((\ell{+}2)/\ell)^d)$.
We may assume $t \ge \ell^d \log(\ell^d)$ so that virtually
all of the sites in $B$ actually get updated.  Letting $U^*_B\sigma$ denote the distribution
that results when $t$ single-site updates are performed on $B$, we have that the consequent decrease
in Hamming distance satisfies

\begin{eqnarray*}
\E\Delta_t & = & \P(u~{\rm is}~{\rm updated}) - \frac1N \sum_{\partial B \owns u} \rho(U_B^*\sigma,U^*_B\sigma') \\
& \ge & \P(B \owns u) - \frac1N \sum_{\partial B \owns u} \left( \rho(U^*_B\sigma,U_B\sigma)
  + \rho(U_B\sigma,U_B\sigma') + \rho(U_B\sigma',U_B^*\sigma') \right) \\
& \ge & \frac{\ell^d}{N^d} - \frac{2|\partial B|}{N^d} \delta - \frac1{N^d} \sum_{\partial B \owns u} \Phi_u(B) \\
& \ge & \frac{\ell^d}{N^d} - \frac{\gamma \ell^d}{2N^d} - (1 - \gamma)\frac{\ell^d}{N^d} = \frac{\gamma \ell^d}{2N^d}~.
\end{eqnarray*}

Suppose next that ${\bf T}$ is a nonnegative-integer-valued random variable that satisfies
$\P({\bf T}<t) < \delta/\ell^d$.  Since the Hamming distance of any two configurations
is bounded by $\ell^d$, if we perform ${\bf T}$ random single-site updates on the block $B$, we get
\begin{eqnarray}\label{eqn:appr}
\E\Delta_{\bf T} & > & \gamma \frac{\ell^d}{2N^d} - \frac{p}{N^d} \sum_{\partial B \owns u} \ell^d \\ \nonumber
& \ge & \frac {\gamma \ell^d}{2N^d} - \frac{\delta}{N^d \ell^d} |\partial B|\ell^d \\ \nonumber
& \ge & \frac {\gamma \ell^d}{4N^d}~.
\end{eqnarray}
so this ``approximate block update'' is still contracting.

Suppose now that we choose $\vec{j} = (j_1,\dots,j_d)$ uniformly at random in $\{0,\dots,\ell\}^d$ and update
(in the normal fashion) all the blocks $B_{\vec{j}+(\ell+1)\vec{k}}$ where $\vec{k} \in \Z^d$.  These blocks are disjoint, and, moreover,
no block has an exterior neighbor belonging to another block, hence it makes no difference in
what order the updates are made.  We call this series of updates a ``global block update,''
and claim that it is contracting---meaning, in this case, that a {\em single} global update
reduces the Hamming distance between any two configurations $\sigma$ and $\tau$ by a constant factor $1{-}\gamma'$.

To see this, we reduce to the case where $\sigma$ and $\tau$ differ only at a vertex $u$ and
average over choice of $\vec{j}$ to get that the expected decrease in Hamming distance is
$$
\frac1{(\ell{+}1)^d}\left( \ell^d - \sum_{\vec{k}} \Phi_u(B_{\vec{j}+(\ell+1)\vec{k}})\right)
$$
which, by comparing with (\ref{eqn:ham}), exceeds $\gamma \big( \ell/(\ell{+}1) \big)^d$.

If the updates of the blocks $B_{\vec{j}+(\ell+1)\vec{k}}$ are of the approximate variety as
described above, we get an ``approximate global block update" which still contracts.

Let us now consider Glauber dynamics (successive updates of random single sites)
for time $2t|V|/\ell^d$, with the object of showing that this will reduce the
expected Hamming distance between any two configurations by at least a constant factor.
The number of updates that hit a particular block $B$ will then be a binomially distributed
random variable ${\bf T}$ with mean $2t$; its probability of falling below $t$ is bounded above
by $e^{-t/4}$ (see, e.g., \cite{AS}, Theorem A.1.13, p.\ 312).  Recall that we took $t \ge \ell^d \log(\ell^d)$;
if $t <4 \log(\ell^d/\delta)$ then we increase $t$ to equal the larger right-hand side, and note that it is still depends only on $\gamma, \ell$ and not on $N$.
We have thus ensured that $\P({\bf T} < t) \le \delta/\ell^d$ as required for (\ref{eqn:appr}).

It follows that if we choose $\vec{j}$ uniformly at random as above and {\em censor} all
updates of sites not in $\bigcup_{\vec{k}} B_{\vec{j}+(\ell+1)\vec{k}}$, then we have achieved an
approximate global block update, and thus a contraction of expected Hamming distance by a factor $1-\gamma/4$.

We deduce that $O(\log|V|)$ approximate global block updates suffice to reduce the maximal Kantorovich distance from its initial value $|V|$ (The Hamming distance between the top and bottom configurations) to any desired small constant. Recall that  Kantorovich distance dominates total variation distance, and each approximate global block update involves $O(|V|)$ single cite updates, with censoring of updates that fall on the (random) boundary. Thus with this censoring, uniformly random single-site updates mix in time O$(|V| \log |V|)$.

By Theorem~\ref{subseq}, censoring these updates cannot improve mixing time,
hence the mixing time for standard single-site Glauber dynamics is again O$(|V| \log |V|)$.
\end{proof}

In the above theorem the periodic boundary and divisibility condition were assumed only
for convenience in the proof, variations of which can be applied in many other settings.
Indeed, since we announced our censoring inequality in 2001,
other applications to block dynamics have been made by Martinelli and Sinclair~\cite{MS}, Martinelli and Toninelli~\cite{MT},
Mossel and Sly~\cite{MoSl}, Ding, Lubetzky and Peres~\cite{DLP}, and Ding and Peres~\cite{DP}.
In particular, \cite{DP} uses the censoring inequality to prove a uniform lower bound asymptotic to
$n \log{n}/4$ for the mixing time of Glauber dynamics of the Ising model on any $n$-vertex graph.

Note that even if the Markov random field is not monotone, our proof shows mixing time
O$(|V| \log |V|)$ for censored single-site dynamics; this improves by a log factor Corollary 3.3
of Van den Berg and Brouwer \cite{BB}.

\section{Proof of the censoring inequality (Theorem~\ref{subseq})} \label{pf-censor}

\begin{lemma}\label{Lemma1}
Let $\langle \Omega,S,V,\pi \rangle$ be a monotone system, let $\mu$ any
distribution on $\Omega$, and let $\mu_v$ be the result of updating $\mu$
at the site $v \in V$.  If $\mu/\pi$ is increasing on $\Omega$, then so is
$\mu_v/\pi$.
\end{lemma}

\begin{proof}
Define $f : S^V \to \R$ by
\begin{equation} \label{deff}
f(\sigma):= \max \Bigl\{\frac{\mu(\omega)}{\pi(\omega)} \, :
  \, \omega \in \Omega, \, \, \omega \le \sigma \Bigr\} \,
%\quad \mbox{ \rm for } \; \sigma \in S^V \,,
\end{equation}
with the convention that $f(\sigma)=0$ if there is no
$\omega \in \Omega$ satisfying $\omega \le \sigma$.
Then $f$ is increasing on $S^V$, and $f$ agrees with $\mu/\pi$ on $\Omega$.

Let $\sigma < \tau$ be two configurations in $\Omega$; we wish to
show that
\begin{equation}
\frac{\mu_v}{\pi}(\sigma) \le \frac{\mu_v}{\pi}(\tau).
\end{equation}

Note first that for any $s \in S$,
$$
f(\sigma_v^s) \le f(\tau_v^s) \,,
$$
since $f$ is increasing.  Furthermore,
$f(\tau_v^s)$ is an increasing function of $s$.
Thus, by (\ref{update}),
\begin{eqnarray} \nonumber
\frac{\mu_v}{\pi}(\sigma) &=&
 \frac{\mu(\sigma_v^\bullet)}{\pi(\sigma_v^\bullet)}
= \sum_{s \in S} f(\sigma_v^s)
\frac{\pi(\sigma_v^s)}{\pi(\sigma_v^\bullet)} \\ [1ex] \nonumber
&\le&  \sum_{s \in S} f(\tau_v^s)
\frac{\pi(\sigma_v^s)}{\pi(\sigma_v^\bullet)} \le
\sum_{s \in S} f(\tau_v^s)
\frac{\pi(\tau_v^s)}{\pi(\tau_v^\bullet)} = \frac{\mu_v}{\pi}(\tau) \,,
\end{eqnarray}
where the last inequality follows from the stochastic domination
guaranteed by monotonicity of the system.
\end{proof}

\begin{lemma}\label{Lemma3}
Suppose that $S$ is totally ordered.  If $\alpha$ and $\beta$
are probability distributions on $S$ such that $\alpha/\beta$ is
increasing on $S$ and $\beta(s)>0$ for all $s \in S$, then
$\alpha \succeq \beta$.
\end{lemma}

\begin{proof}
Let $g$ be any increasing function on $S$; then, with all sums taken
over  $s \in S$,
$$
\sum g(s)\alpha(s) = \sum g(s) \frac{\alpha(s)}{\beta(s)}\beta(s)
\ge \sum g(s)\beta(s) \cdot \sum \frac{\alpha(s)}{\beta(s)}\beta(s)
= \sum g(s)\beta(s),
$$
confirming stochastic domination.  The inequality in the chain is
the positive correlations property of totally ordered sets
(which goes back to Chebyshev, see \cite{Li} \S II.2),
 applied to the increasing functions $g$ and $\alpha/\beta$
on $S$ with measure $\beta$.
\end{proof}

\begin{lemma}\label{Lemma4}
Let $\langle \Omega,S,V,\pi \rangle$ be a monotone system.
% with $S$ is totally ordered.
If $\mu$ is a distribution on $\Omega$ such that $\mu/\pi$ is increasing,
then $\mu \succeq \mu_v$
for any $v \in V$.
\end{lemma}

\begin{proof}
Let $g$ be increasing.  If $\sigma \in \Omega$ satisfies
$\mu(\sigma_v^\bullet)>0$, then
$\mu/\mu_v$ is increasing on $\sigma_v^\bullet$.
By Lemma~\ref{Lemma3}
(applied to $\{ s \in S: \, \sigma^s_v \in \Omega \}$ in place of $S$),
for such $\sigma$ we have
$$
\sum_{s \in S} g(\sigma^s_v) \frac{\mu(\sigma^s_v)}{\mu(\sigma_v^\bullet)}
\ge \sum_{s \in S} g(\sigma^s_v)
\frac{\mu_v(\sigma^s_v)}{\mu(\sigma_v^\bullet)}~.
$$
Multiplying by $\mu(\sigma_v^\bullet)$ and
summing over all choices of $\sigma_v^\bullet$
gives
$$
\sum_{\sigma \in \Omega} g(\sigma)\mu(\sigma)
\ge \sum_{\sigma \in \Omega} g(\sigma)\mu_v(\sigma) \,,
$$
establishing the required stochastic dominance.
\end{proof}

\begin{lemma}\label{Lemma-var}
Let $\langle \Omega,S,V,\pi \rangle$ be a monotone system, and
let $\mu,\nu$ be
two arbitrary
distributions on $\Omega$.
If $\nu/\pi$ is increasing on $\Omega$
and $\nu \preceq \mu$, then $\|\nu-\pi\| \le   \|\mu-\pi\|$.
\end{lemma}

\begin{proof}
Let $A=\{\sigma: \, \nu(\sigma)>\pi(\sigma)\}$.
Then the indicator of $A$ is increasing, so
$$
\|\nu-\pi\| = \sum_{\sigma \in A} (\nu(\sigma)-\pi(\sigma)) = \nu(A)-\pi(A)
\le \mu(A)-\pi(A) \,,
$$
since  $\nu \preceq \mu$. The right-hand side is at most $\|\mu-\pi\|$.
\end{proof}

\begin{theorem} \label{onestep}
Let $\langle \Omega,S,V,\pi \rangle$ be a monotone system,
% with totally ordered finite spin set $S$.
Let $\mu$
be the distribution on $\Omega$ which results from successive updates
at sites $u_1,\dots,u_k$, beginning at the top configuration.
Let $\nu$ be defined similarly but with the update at $u_j$ left out.
Then
\begin{enumerate}
\item $\mu \preceq \nu$, and
\item $\|\mu - \pi\| \leq \|\nu - \pi\|$.
\end{enumerate}
\end{theorem}

\begin{proof}
Let $\mu^0$ be the distribution concentrated at the
top configuration, and $\mu^{i} = (\mu^{i-1})_{u_{i}}$ for $i \ge 1$.
Applying Lemma~\ref{Lemma1} inductively, we have that
each $\mu^i/\pi$ is increasing, for $0 \le i \le k$.
In particular, we see from Lemma~\ref{Lemma4} that $\mu^{j-1}
\succeq (\mu_{j-1})_{u_j} = \mu_j$.

If we define $\nu^i$ in the same manner as $\mu_i$, except
that $\nu^j=\nu^{j-1}$, then because stochastic dominance
persists under updates, we have $\nu^i \succeq \mu^i$ for
all $i$; when $i=k$, we get $\mu \preceq \nu$ as desired.

For the second statement of the theorem, we merely apply
Lemma~\ref{Lemma-var}, noting that $\nu^k/\pi$ is increasing
by the same inductive argument used for $\mu$.
\end{proof}

\begin{proof}[\bf Proof of Theorem \ref{subseq}]
Apply Theorem \ref{onestep} inductively, censoring one site at a time.
This establishes the case where the update locations are deterministic.
In the case where the update sequence $v_1(\xi) \ldots, v_m(\xi)$ that yields $\mu$
is random (defined on some probability space $(\Xi, \P_\Xi)$) and its subsequence leading to $\nu$ is also random
(defined on the same probability space), then conditioning on $\xi$ yields measures
$\mu(\xi)$ and $\nu(\xi)$ such that $\mu(\xi) \preceq \nu(\xi)$ and $\nu(\xi)/\pi$ is increasing on $\Omega$.
 These properties are preserved under averaging over $\Xi$, so we conclude that $\mu \preceq \nu$ and $\nu/\pi$
 is increasing on $\Omega$. The inequality between total variation norms follows from Lemma~\ref{Lemma-var}.

\end{proof}

%\begin{figure}[htp]
%\centerline{\epsfig{figure=fig1.ps}}
%\caption{This is just to mark the place for a figure}\label{fig1}
%\end{figure}

\section{Comparison of single site update schemes}

In practice, updates on a system $\langle \Omega,S,V,\pi \rangle$ are often performed
systematically rather than at random.  Typically a permutation of $V$ is fixed
and sites are updated periodically in permutation order.  If the interaction graph is
bipartite, it is possible and often convenient to update all odd sites simultaneously,
then all even sites, and repeat; we call this {\em alternating} updates.  To be fair,
we count a full round of alternating updates as $n$ single updates, so that alternating
updates constitute a special case of systematic updating.

Mixing time may differ from one update scheme to another; for example, if there
are no interactions (so that one update per site produces perfect mixing) then
systematic updating is faster by a factor of $\frac{1}{2}\log n$ than uniformly random updates,
since  after $(\frac{1}{2}-\ep)n \log n$ random updates about $n^{1/2+\ep}$ sites have not been hit,
so counting the number of sites that still have the initial spin implies the total variation distance to equilibrium is still close to 1.
(For a more general $\Omega(n\log n)$ lower bound for Glauber dynamics with random updates see \cite{HS}).

Embarrassingly, there are only a few results to support the observation that mixing times
for the various update schemes never seem to differ by more than a factor of $\log n$
and rarely by more than a constant.  (See \cite{DGJ1, DGJ2} for some recent progress
in the Dobrushin uniqueness regime.)
Theorem~\ref{subseq} allows us to obtain some useful
comparison results for monotone systems, but is still well short of what is suspected to be true.

\begin{theorem}\label{alt->sys}
Let ${\mathcal A}$ be the alternating update scheme, and ${\mathcal S}$ an arbitrary
systematic update scheme, for a bipartite monotone system $\langle \Omega,S,V,\pi \rangle$.
Then the mixing time for ${\mathcal S}$ (starting at the top state) is no more than twice
the mixing time for ${\mathcal A}$.
\end{theorem}

\begin{proof}
When updating according to ${\mathcal S}$, we censor all even-site updates; on even passes,
all odd-site updates.  Since successive updates of sites of the same parity commute, the
result is exactly ${\mathcal A}$ and an application of Theorem~\ref{subseq} shows that we
mix at a cost of at most a factor of 2.
\end{proof}

\begin{theorem}\label{alt->rand}
Let ${\mathcal A}$ be the alternating update scheme, and ${\mathcal R}$ the uniformly random
update scheme, for a bipartite monotone system $\langle \Omega,S,V,\pi \rangle$. Then the
mixing time for ${\mathcal R}$ (starting at the top state) is no more than $2 \log n$ times
the mixing time for ${\mathcal A}$.
\end{theorem}

\begin{proof}
When updating according to ${\mathcal R}$, we censor all even-site updates until all
odd sites are hit; then we censor all odd-site updates until all even sites are hit, and
repeat.  Since each of these steps takes $2(n/2)\log(n/2)$ updates on average,
Theorem~\ref{subseq} guarantees a loss of at most a factor of $2 \log n$.
\end{proof}

\begin{theorem}\label{rand->sys}
Let ${\mathcal R}$ be the uniformly random update scheme, and ${\mathcal S}$ an arbitrary
systematic update scheme, for a monotone system $\langle \Omega,S,V,\pi \rangle$ of maximum
degree $\Delta_{\rm max}$. Then the mixing time for ${\mathcal S}$ is no more than O$(\sqrt{\Delta_{\rm max} n})$ times
the mixing time for ${\mathcal R}$.
\end{theorem}

\begin{proof}
Prior to implementing a round of ${\mathcal S}$, we choose uniformly random sites one by one
as long as no two are adjacent; since the probability of adjacency for a random pair of
sites is at most $(\Delta_{\rm max}+1)/n$, this ``birthday problem'' procedure will keep about
$\sqrt(n/\Delta_{\rm max})$ updates.  All updates of sites not on this list are censored from the
upcoming round of ${\mathcal S}$, incurring a loss of a factor of $n/(\sqrt(n/\Delta_{\rm max})) =
\sqrt{\Delta_{\rm max} n}$.  Since updates of non-adjacent sites commute, Theorem~\ref{subseq} applies.
\end{proof}

If $\langle \Omega,S,V,\pi \rangle$ is bipartite, then since the alternating scheme is a
systematic scheme, Theorem~\ref{rand->sys} applies to it as well.

From systematic updates to alternating or random updates, there seems to be nothing better
to do in our context than to score one update per systematic round, incurring a factor
of $n$ penalty.  %These results are illustrated below in Figure~\label{fig:compare}.

\subsection{Hanging subgraphs}

Let $H$ be a subgraph of the finite graph $G$, on which some system $\langle \Omega,S,V,\pi \rangle$
is defined, and suppose what is wanted is mixing on $H$.  When continuous-time Glauber dynamics
is employed, it is natural to compare mixing time $T_H$ on $H$ by itself (that is, with the rest of $G$
destroyed) with mixing time $T_G$  when all points of $G$ are being updated.  Indeed, for the
Ising model (with no external field), we conjecture that $T_H$ never exceeds $T_{G|H}$---echoing
a conjecture of the first author for spectral gaps, cited in \cite{Nacu} and
proved there when $G$ is a cycle.  Putting it another way, we think {\em bigger is slower}.

Because the Ising model is a Markov random field, and its stationary
distribution on a single site is independent of the graph, it enjoys the following property:
if only one vertex (say, $x$) of $H$ is adjacent to vertices of $G \setminus H$, then the stationary
distribution on $H$ is identical to the stationary distribution on $G$ restricted to $H$.
To see this, it suffices to note that either stationary distribution can be obtained
by flipping a coin to determine the sign of $x$, then conditioning the rest of the configuration
on the result.

We can now make use of Theorem~\ref{subseq}, together with monotonicity of the Ising model, to
prove our conjecture in this limited case.

\begin{theorem}\label{hanging} Let $H$ be a subgraph of the finite graph $G$ and suppose that at most one
vertex of $H$ is adjacent to vertices of $G \setminus H$.  Begin in the all ``$+$'' state and
fix a mixing tolerance $\ep$ for continuous Glauber dynamics.  Then $T_H(\ep)
\le T_{G|H}(\ep)$.
\end{theorem}

\begin{proof}
The result is of course trivial if $H$ is disconnected from $G \setminus H$; otherwise
let $x$ be the unique vertex of $H$ with neighbors outside $H$. Let $Q = \langle v_1,\dots,v_k \rangle$
be the target sites of a sequence of updates on $H$, and let $Q'$ on $G$ be the result of replacing
each update of $x$ in $Q$ by a block update of $\{x\} \cup (G \setminus H)$.  Then, on account
of the property noted above, the effects of $Q$ and $Q'$ are identical on $H$.

If it were not the case that $T_H(\ep) \le T_{G|H}(\ep)$, then there would in particular be
an update sequence $Q$ for $H$ and a supersequence $Q^+$ for $G$, all added sites being outside $H$,
such that $Q^+$ gets $H$ closer by some $\delta>0$ to stationarity than does $Q$.  However, for large
enough $j$, we can replace the block updates in $Q'$ by $j$ single-site updates within $\{x\} \cup
(G \setminus H)$ to get a new update sequence $Q''$ which contains $Q^+$, but whose resulting
distribution matches that of $Q'$ (thus also $Q$) to within total variation $\delta/2$.  This would
force $Q''$ to mix better than $Q^+$, contradicting Theorem~\ref{subseq}.
\end{proof}

\noindent{\bf Acknowledgements} we are grateful to Dana Randall for telling us about the censoring idea of
Van den Berg and Brouwer and for suggesting the question addressed in this paper.
We also thank the participants of the 2001 IES in Sweden, especially Svante Janson and Russell Lyons,
for helpful discussions during the early stages of this work.

\end{document}